%% file: main.tex
\pdfoutput=1
\documentclass[letterpaper,10pt,conference]{ieeeconf}

\IEEEoverridecommandlockouts
\let\labelindent\relax
\include{macros-abbrev}
\usepackage{multicol}
\usepackage{graphicx,epsfig,float,wrapfig}
\usepackage{epstopdf}
\usepackage{float}
\usepackage[left=2cm, right=2cm, top=2cm, bottom=2.5cm]{geometry}
\usepackage{epsfig}
\usepackage{url}
\usepackage{amssymb,amsmath,amsfonts,layout}
\usepackage{makeidx}
\usepackage{blkarray,multirow}
\usepackage{cite}
\usepackage{enumerate}
\usepackage{algpseudocode}
\usepackage{algorithm}
\usepackage{enumitem,indentfirst}

\usepackage{amsthm}
\usepackage{bbm}
\usepackage{xspace}
\usepackage [english]{babel}
\usepackage [autostyle, english = american]{csquotes}
\MakeOuterQuote{"}

\newcommand{\nulo}{\operatorname{null}}

\algnewcommand{\Inputs}[1]{%
	\State \textbf{Input:}
	\Statex \hspace*{\algorithmicindent}\parbox[t]{.8\linewidth}{\raggedright #1}
}
\algnewcommand{\Initialize}[1]{%
	\State \textbf{Initialize:}
	\Statex \hspace*{\algorithmicindent}\parbox[t]{.8\linewidth}{\raggedright #1}
}
\algnewcommand{\Outputs}[1]{%
	\State \textbf{Output:}
	\Statex \hspace*{\algorithmicindent}\parbox[t]{.8\linewidth}{\raggedright #1}
}

\newcommand{\distnewton}{\textsc{distributed approx-Newton}\xspace}
\newcommand{\dana}{\textsc{DANA}\xspace}

\usepackage{tikz}
\usetikzlibrary{shapes,arrows,positioning} 
\tikzset{
	>=stealth',
	block/.style={
		rectangle,
		rounded corners,
		draw=black, very thick,
		text width=12em,
		minimum height=3em,
		text centered},
	link/.style={
		->,
		thick,
		shorten <=2pt,
		shorten >=2pt},
	decision/.style={
		diamond,
		draw, very thick,
		fill=blue!20, 
		text width=8em,
		aspect=3,
		text centered}
}

\newcommand{\longthmtitle}[1]{\mbox{}{\bf \textit{(#1).}}}
\theoremstyle{plain}
\newtheorem{thm}{Theorem}
\newtheorem{lem}{Lemma}

\theoremstyle{definition}
\newtheorem{assump}{Assumption}

\theoremstyle{remark}

\makeatletter
\newcommand{\pushright}[1]{\ifmeasuring@#1\else\omit\hfill$\displaystyle#1$\fi\ignorespaces}
\newcommand{\pushleft}[1]{\ifmeasuring@#1\else\omit$\displaystyle#1$\hfill\fi\ignorespaces}
\renewcommand*\env@matrix[1][*\c@MaxMatrixCols c]{%
	\hskip -\arraycolsep
	\let\@ifnextchar\new@ifnextchar
	\array{#1}}
\makeatother

\DeclareMathOperator{\N}{\mathcal{N}}

\newcommand{\nodes}{\N}

\newcommand\oprocendsymbol{\hbox{$\bullet$}}
\newcommand\oprocend{\relax\ifmmode\else\unskip\hfill\fi\oprocendsymbol}

\begin{document}
	
	\title{Weight Design of Distributed Approximate Newton Algorithms for Constrained Optimization}
	\author{Tor Anderson \quad Chin-Yao Chang \quad Sonia Mart{\'\i}nez
		\thanks{Tor Anderson, Chin-Yao Chang,  and Sonia Mart{\'\i}nez are
			with the Department of Mechanical and Aerospace Engineering,
			University of California, San Diego, CA, USA. Email: {\small {\tt
					\{tka001, chc433, soniamd\}@eng.ucsd.edu}}} }

	\maketitle
	
	\begin{abstract}
		% This paper presents a distributed algorithm to solve an
		% economic dispatch problem, which takes the form of a
		% linearly-constrained resource allocation
		% problem. Distributed gradient-based methods are commonly
		% used to solve problems of this form, which inherit slow
		% convergence.  The Newton method is a centralized alternative
		% which uses second-order information to provide faster
		% convergence.  However, computing a Newton step is difficult
		% in distributed settings and typically requires all-to-all
		% agent communication.  In this paper, we propose the
		% $\distnewton$ algorithm to approximate the Newton step with
		% only distributed communication. The convergence of this
		% algorithm is discussed and rigorously analyzed.  In
		% addition, we aim to address the problem of designing
		% communication topologies and weightings that are optimal for
		% second-order methods.  To this end, we propose an effective
		% approximation which is loosely based on completing the
		% square to address the NP-hard bilinear optimization involved
		% in the design.  Simulations demonstrate that our proposed
		% weight design applied to the $\distnewton$ algorithm has a
		% superior convergence property compared to existing
		% gradient-inspired weight design applied to the Distributed
		% Gradient Descent method.
		Motivated by economic dispatch and linearly-constrained
		resource allocation problems, this paper proposes a novel
		$\distnewton$ algorithm that approximates the standard
		Newton optimization method. A main property of this
		distributed algorithm is that it only requires agents to
		exchange constant-size communication messages. The
		convergence of this algorithm is discussed and rigorously
		analyzed.  In addition, we aim to address the problem of
		designing communication topologies and weightings that are
		optimal for second-order methods.  To this end, we propose
		an effective approximation which is loosely based on
		completing the square to address the NP-hard bilinear
		optimization involved in the design.  Simulations
		demonstrate that our proposed weight design applied to the
		$\distnewton$ algorithm has a superior convergence property
		compared to existing weighted and distributed first-order
		gradient descent methods.
	\end{abstract}
	
	\IEEEpeerreviewmaketitle

	\section{Introduction}

	\textit{Motivation.} 
	Networked systems endowed with distributed, multi-agent
	intelligence are becoming pervasive in modern infrastructure
	systems such as power networks, traffic networks, and
	large-scale distribution systems. However, these advancements
	lead to new challenges in the coordination of the multiple
	agents operating the network, which are mindful of the network
	dynamics, availability of partial information, and
	communication constraints. Distributed algorithms are useful
	in managing such systems in a manner that elegantly utilizes
	network resources, which are often (by nature) non-centralized
	but more scalable. To this end, distributed convex
	optimization is a rapidly emerging field which seeks to
	develop such algorithms and understand their properties,
	e.g.~convergence or robustness, from a mathematically rigorous
	perspective.  There are a wide array of techniques that seek
	to distributively compute the solution to optimization
	problems. One such problem is economic dispatch, in which a
	total net load constraint must be satisfied by a set of
	generators which each have an associated cost of producing
	electricity. This problem has gained a large amount of recent
	attention due to the rapid emergence of distributed renewable
	energy resources, such as wind and solar power and demand
	response.  However, the existing distributed techniques are
	often limited by speed of convergence.  Motivated by this and
	resource allocation problems in general, we investigate the
	design of topology weighting strategies that build on the
	Newton method and lead to improved convergence speeds.

	\textit{Literature Review.}  The Newton method for minimizing
	a real-valued multivariate objective function is well
	characterized for centralized contexts in~\cite{SB-LV:04}. The
	notion of computing an \emph{approximate} Newton direction in
	distributed contexts has gained popularity in recent
	literature, such as~\cite{AM-QL-AR:15}
	and~\cite{AJ-AO-MZ:09}. In the former work, the authors
	propose a method which uses the Taylor series expansion for
	inverting matrices. However, it assumes that each agent keeps
	an estimate of the entire decision variable, which does not
	scale well in problems where this variable dimension is equal
	to the number of agents in the network. Additionally, the
	optimization is unconstrained, which helps to keep the network
	decoupled but has a narrower scope in practice.  The latter
	work poses a separable optimization with an equality
	constraint characterized by the incidence matrix. The proposed
	approximated Newton method may be not directly applied to
	networks with constraints that involve the information of all
	agents.  
	In addition to the aforementioned works, the
	papers~\cite{DJ-JX-JM-14,FZ-DV-AC-GP-LS:13,RC-GN-LS-DV:15}
	incorporate multi-timescaled dynamics together with a dynamic
	or online consensus step to speed up the convergence of the
	agreement subroutine. 
	The aforementioned works only consider uniform edge weight,
	while sophisticated design of the weighting may improve the
	convergence. In~\cite{LX-SB:06}, the Laplacian weight design
	problem for separable resource allocation is approached from a
	Distributed Gradient Descent perspective. Solution
	post-scaling is also presented, which can be found similarly
	in~\cite{MM-RT:07} and~\cite{YS:03} for improving the
	convergence of the Taylor series expression for matrix
	inverses.  In~\cite{SYS-MA-LEG:10}, the authors consider edge
	weight design to minimize the spectrum of Laplacian
	matrices. However, in the Newton descent framework, the weight
	design problem formulates as an NP-hard bilinear problem,
	which is challenging to solve.
	Overall, the current weight-design techniques that are
	computable in polynomial time are only mindful of first-order
	algorithm dynamics. This approach has its challenges, which
	manifest themselves in a bilinear design problem and more
	demanding communication requirements, but using second-order
	information is more heedful of the problem geometry and leads
	to faster convergence speeds.

	\textit{Statement of Contributions.}  In this paper, we
	propose a novel framework to design a weighted Laplacian
	matrix that is used in the solution to a multi-agent
	optimization problem via sparse approximated Newton
	algorithms. Motivated by economic dispatch, we start by
	formulating a separable resource allocation problem subject to
	a global linear constraint, and then derive an equivalent
	unconstrained form by means of a Laplacian matrix. The Newton
	steps associated with the unconstrained optimization problem
	do not inherit the same sparsity as the distributed
	communication network. To address this issue, we consider
	approximations based on a Taylor series expansion, where the
	first few terms inherit certain level of sparsity as
	prescribed by choice of the Laplacian matrix.  We analyze the
	approximated algorithms and show their (exponential)
	convergence for any truncation of the series expansion.  The
	effectiveness of the approximation heavily depends on the
	selection of the weighted Laplacian, which we then set out to
	optimize. The optimal solution requires solving a NP-hard
	bilinear optimization problem, which we address by proposing a
	novel approach which systematically computes an effective
	solution in polynomial time. A bound on the \emph{best-case}
	solution of the original bilinear problem is also
	given. Furthermore, through a formal statement of the proposed
	$\distnewton$ algorithm (or $\dana$), we find several
	interesting insights on second-order distributed methods. We
	compare the results of our design and algorithm to a generic
	weighting design of Distributed Gradient Descent (DGD)
	implementations in simulation. Our weighting design shows
	superior convergence to DGD.

	\textit{Organization.} The rest of the paper is organized as
	follows. Section~\ref{sec:prelim} introduces the notations and
	fundamentals used in this paper. We formulate the optimal
	resource allocation problem in
	Section~\ref{sec:prob-statement}. In Section~\ref{sec:main},
	we propose a distributed algorithm that approximates the
	Newton step in solving the
	optimization. Section~\ref{sec:sims-discuss} demonstrates the
	effectiveness of the proposed algorithm. We conclude the paper
	in Section~\ref{sec:conclusion}.

	\section{Preliminaries}
	\label{sec:prelim}
	\subsection{Notation}
	Let $\real$ and $\real_{+}$ denote the set of real and positive real
	numbers, respectively, and let $\natural$ denote the set of natural
	numbers.  For a vector ${x \in \real^n}$, we denote by $x_i$ the
	$\supscr{i}{th}$ entry of $x$.  For a matrix ${A \in\real^{n \times
			m}}$, we write $A_i$ as the $\supscr{i}{th}$ row of $A$ and
	$A_{ij}$ as the element in the $\supscr{i}{th}$ row and
	$\supscr{j}{th}$ column of $A$. Discrete time-indexed variables are
	written as $x^k$, where $k$ denotes the current time step. The
	transpose of a vector or matrix is denoted by $x^\top$ and $A^\top$,
	respectively. We use the shorthand notations ${\mathbf{1}_n = [1,
		\dots, 1]^\top \in \real^n}$, ${\mathbf{0}_n = [0, \dots, 0]^\top
		\in \real^n}$, and $I_n$ to denote the ${n \times n}$ identity
	matrix. The standard inner product of two vectors $x,y \in \real^n$ is
	written $\langle x, y\rangle$, and $x \perp y$ indicates $\langle x,
	y\rangle = 0$. For a real-valued function ${f : \real^n \rightarrow
		\real}$, the gradient vector of $f$ with respect to $x$ is denoted
	by ${\nabla_x f(x)}$ and the Hessian matrix with respect to $x$ by
	${\nabla_x^2 f(x)}$.  The positive (semi) definiteness and negative
	(semi) definiteness of a matrix ${A \in \real^{n \times n}}$ is
	indicated by ${A \succ 0}$ and ${A \prec 0}$ (resp. ${A \succeq 0}$
	and ${A \preceq 0}$). The set of eigenvalues of a symmetric matrix ${A
		\in \real^{n \times n}}$ is ordered as $\lambda_1 \leq \dots \leq
	\lambda_n$ with associated eigenvectors $v_1, \dots , v_n \in
	\real^n$.  An orthogonal matrix ${T \in \real^{n \times n}}$ has the
	property ${T^\top T = T T^\top = I_n}$ and ${T^\top = T^{-1}}$. For a
	finite set $\mathcal{S}$, $\vert \mathcal{S} \vert$ is the cardinality
	of the set. The uniform distribution on the interval $[a, b]$ is
	indicated by $\mathcal{U} [a, b]$.
	
	\subsection{Graph Theory}
	A network of agents is represented by a graph $\mathcal{G} =
	(\nodes,\mathcal{E})$, with a node set $\nodes = \until{n}$, and edge
	set $\mathcal{E} \subset \nodes \times \nodes$. The edge set
	$\mathcal{E}$ has elements $(i,j) \in \mathcal{E}$ for $j \in
	\nodes_i$, where $\nodes_i \subset \nodes$ is the set of out-neighbors
	of agent $i \in \nodes$. The union of out-neighbors to each agent $j
	\in \nodes_i$ are the 2-hop neighbors of agent $i$, and denoted by
	$\nodes_i^2$. More generally, $\nodes_i^p$, or set of $p$-hop
	neighbors of $i$, is the union of the neighbors of agents in
	$\nodes_i^{p-1}$.  From this point forward, we consider undirected
	networks, so ${(i,j) \in \mathcal{E}}$ if and only if ${(j,i) \in
		\mathcal{E}}$.  The graph has an associated weighted Laplacian ${L
		\in \real^{n \times n}}$, defined as
	\begin{equation*}
	L_{ij} =
	\begin{cases}
	-w_{ij}, & j \in \nodes_i, j \neq i, \\
	w_{ii}, & j = i, \\
	0, & \text{otherwise},
	\end{cases}
	\end{equation*}
	with weights $w_{ij} = w_{ji} > 0, \forall j \neq i$, and total
	incident weight $w_{ii}$ on~$i\in \nodes$, ${w_{ii} = \sum_{j \in
			\nodes_i} w_{ij}}$. Evidently, $L$ has an eigenvector ${v_1 =
		\mathbf{1}_n}$ with an associated eigenvalue ${\lambda_1 = 0}$, and
	${L = L^\top \succeq 0}$. The graph is connected iif $0$ is a simple
	eigenvalue, i.e.~$0 = \lambda_1 < \lambda_2 \leq \dots \leq
	\lambda_n$.
	
	A Laplacian $L$ can also be written as a product of its incidence
	matrix ${E \in \{-1,0,1\}^{\vert \mathcal{E} \vert \times n}}$ and a
	diagonal matrix ${X \in \real_{+}^{\vert \mathcal{E} \vert \times
			\vert \mathcal{E} \vert}}$, whose entries are the weights
	$w_{ij}$. Each row of $E$ is associated with an edge, which can be
	alternatively enumerated as $e \in \until{\vert \mathcal{E} \vert}$,
	where the elements $E_{ev}$ are given by
	\begin{equation*}
	E_{ev} =
	\begin{cases}
	1, & v = i,			\\
	-1, & v = j,			\\
	0, & \text{otherwise}.
	\end{cases}
	\end{equation*}
	Then, $L$ can be written as $L = E^\top X E$.
	
	\subsection{Schur Complement}
	The following lemma will be used in the sequel.
	\begin{lem}[\cite{FZ:05}]\longthmtitle{Matrix Definiteness via Schur Complement}\label{lem:schur-comp}
		Consider a symmetric matrix $M$ of the form
		\begin{equation*}
		M = \begin{bmatrix}
		A & B \\ B^\top & C
		\end{bmatrix}.
		\end{equation*}
		If $C$ is invertible, then the following properties hold: \\
		(1) $M \succ 0$ if and only if $C \succ 0$ and $A - BC^{-1}B^\top \succ 0$. \\
		(2) If $C \succ 0$, then $M \succeq 0$ if and only if $A - BC^{-1}B^\top \succeq 0$.
	\end{lem}
	\subsection{Taylor Series Expansion for Matrix Inverses} \label{ssec:Taylor}
	
	A full-rank matrix $A \in \real^{n \times n}$ has a matrix inverse,
	$A^{-1}$, which is characterized by the relation $AA^{-1} = I_n$. In
	general, it is not straightforward to compute this inverse via a
	distributed algorithm.  However, if the eigenvalues of $A$ satisfy
	$\vert 1 - \lambda_i(A) \vert < 1, \forall \,i \in \mathcal{N}$, then
	we can employ the Taylor expansion to compute its inverse as follows:
	\begin{align*}
	A^{-1} = \sum_{p=0}^\infty (I_n - A)^p.
	\end{align*}
	Note that, if the sparsity structure of $A$ represents a network
	topology, then traditional matrix inversion techniques such as
	Gauss-Jordan elimination still necessitate all-to-all
	communication. However, agents can communicate and compute locally to
	obtain each term in the previous expansion.  It can be seen via the
	diagonalization of $I_n - A$ that the terms of the sum become small as
	$p$ increases due to the assumption on the eigenvalues of $A$. The
	convergence of these terms is exponential and limited by the slowest
	converging mode, i.e.~$\max{\vert 1 - \lambda_i(A) \vert}$.
	
	We can compute an approximation of $A^{-1}$, denoted by
	$\widetilde{A^{-1}_q}$, in finite steps by computing and summing the
	terms up to the $\supscr{q}{th}$ power.  We refer to this
	approximation as a \textit{q-approximation} of $A^{-1}$.
	
	\section{Problem Statement} \label{sec:prob-statement} Motivated by
	the economic dispatch problem, in this section, we pose the separable
	resource allocation that we aim to solve distributively. We
	reformulate it as an unconstrained optimization problem whose decision
	variable is in the span of the graph Laplacian, and motivate the
	characterization of a second-order Newton-inspired method.
	
	Consider a group $\nodes$ of agents or nodes, indexed by $i \in
	\nodes$, each associated with a local convex cost function $f_i :
	\real \rightarrow \real$, $i \in \nodes$. These agents can be thought
	of as generators in an electricity market, where each function
	argument, $x_i \in \real$, $i \in \nodes$, represents the power in
	megawatts that agent $i$ produces at a cost characterized by
	$f_i$. The economic dispatch problem aims to satisfy a global
	load-balancing constraint $\sum_{i=1}^n x_i = d$ for minimal global
	cost $f : \real^n \rightarrow \real$, where $d$ is the total load
	consumption in megawatts. We relax the full problem by excluding the
	box constraints on the states $x_i$. Then, the relaxed economic dispatch
	optimization problem we are interested in solving is stated as:
	\begin{alignat}{3}
	\Pc 1: \quad 
	&\underset{x}{\min} \quad &&f(x) = \sum_{i=1}^n f_i(x_i) \label{eq:main-cost} \\
	&\text{s.t.} &&\sum_{i=1}^n x_i = d. \label{eq:main-constraint}
	\end{alignat}
	Excluding the box constraints has the effect of allowing some states
	to be negative, which physically could addressed via curtailment of
	some baseline loads.  For simplicity, we do not consider these
	constraints here; addressing this question via quadratic loss
	functions is the subject of current work.  Distributed optimization
	algorithms based on a gradient descent approach are available to solve
	$\Pc 1$~\cite{MZ-SM:15-book}.  However, by only taking into account
	first-order information of the cost functions, these methods tend to
	be inherently slow. As for a Newton (second-order) method, the
	computation of the descent direction is not distributed. To see this,
	recall the unconstrainted Newton step defined as ${\Delta
		\subscr{x}{nt} := -\nabla^2 f(x)^{-1}\nabla f(x)}$, see
	e.g.~\cite{SB-LV:04}. In this context, the equality constraint can be
	eliminated by imposing ${x_n = d - \sum_{i=1}^{n-1} x_i}$. Then,
	\eqref{eq:main-cost} becomes ${f(x) = \sum_{i=1}^{n-1} f_i(x_i) +
		f_n(d - \sum_{i=1}^{n-1} x_i)}$. In general, the resulting Hessian
	$\nabla^2_x f(x)$ is fully populated and its inverse requires
	all-to-all communication among agents in order to compute the
	second-order descent direction.

	Instead, we eliminate \eqref{eq:main-constraint} by introducing a 
	network topology as encoded by a Laplacian matrix $L$ associated with
	a connected graph $\graph$, and an initial condition $x^0 \in
	\real^n$, with the following assumptions.
	\begin{assump}\longthmtitle{Symmetric Laplacian matrix} 	\label{ass:conn-graph}
		The weighted graph characterized by $L$ is undirected and connected,
		i.e. $L = L^\top$ and $0$ is a simple eigenvalue of $L$.
	\end{assump}
	\begin{assump}\longthmtitle{Feasible initial condition} \label{ass:initlal}
		The initial state $x^0$ is a feasible point of $\Pc 1$, i.e.
		\begin{equation*}
		\sum_{i=1}^n x_i^0 = d.   
		\end{equation*}
	\end{assump}
	The objective and constraint ~\eqref{eq:main-cost}--\eqref{eq:main-constraint} 
	can now be restated as the following equivalent problem:
	\begin{align*}
	\Pc 2: \quad & \underset{z}{\min}
	\quad f(x^0 + Lz) = \sum_{i=1}^n \textit{f}_i(x_i^0 + L_i z). %\label{eq:laplacian-problem}
	\end{align*}
	Using the property that $\mathbf{1}_n$ is an eigenvector associated
	with $\lambda_1 = 0$ we have that $\mathbf{1}_n^\top (x^0 + Lz) = d$
	for any $z \in \real^n$. Then, the solution to
	$\Pc 2$ is a feasible and optimal solution to
	$\Pc 1$. An analogous formulation of equality constrained
	Newton descent for centralized solvers is given in~\cite{SB-LV:04}; 
	in our distributed framework, the row
	space of the Laplacian is a useful property to
	address~\eqref{eq:main-constraint}.
	
	We aim to leverage the freedom given by the elements of $L$ in order
	to compute an approximate Newton direction to~$\Pc
	2$, 
	denoted by $\subscr{\tilde{z}}{nt}$. In every iteration $k$, the
	approximated Newton step $\subscr{\tilde{z}}{nt}$ is used to update
	$x^k$ by the following
	
	\begin{align} \label{eq:short algorithm} & x^{k+1} = x^k +
	\alpha L \subscr{\tilde{z}}{nt},
	%		\end{split}
	\end{align}
	where $\alpha > 0$ is a step size. There is agent-to-agent
	communication and local computations embedded in this informal
	statement, which will be described and restated formally in
	the following section.
	
	\section{Main Results}
	\label{sec:main}
	
	In this section, we characterize an \emph{approximate} Newton
	step that can be computed distributively to solve problem $\Pc
	1$. We then pose the weight-design problem and develop an
	approximation which enables a solution to be computed in
	polynomial time. Next, we provide a method for estimating the
	best-case scenario of the solution to the NP-hard problem. We
	then formally state the $\distnewton$ algorithm and rigorously
	analyze its convergence properties.
	
	\subsection{Characterization of the Approximate Newton
		Step}\label{ssec:char_approx_newton}
	
	First, note that the Hessian of~\eqref{eq:main-cost}, $H := \nabla_x^2
	f(x)$, is diagonal. In addition, we adopt the following assumption.
	\begin{assump}\longthmtitle{Quadratic cost functions}\label{ass:cost}
		The local costs can be written in the form $f_i(x_i) =
		\dfrac{1}{2}a_i x_i^2 + b_i x_i$ with $a_i \in \real_{+},
		b_i \in \real$. In light of this, given $\delta = \min(a_i),
		\gamma = \max(a_i)$, then
		\begin{equation*}
		0 \prec \delta I_n \preceq \nabla_x^2 f(x) = H \preceq \gamma I_n.
		\end{equation*}
	\end{assump}
	This ensures $H = \diag{a_i}$ is constant, which is necessary
	to construct the notion of an optimal $L$. In fact, this is a
	widely accepted model for generator costs, and the assumption
	is common in the literature for addressing economic
	dispatch. We use $b \in \real^n$ to refer to the vector of
	linear coefficients.
	
	The Hessian of the unconstrained problem with respect to $z$
	can then be computed as $\nabla_z^2 f(x_{0} + L z) = L H
	L$. We would like to find an analogous Newton step in $z$ to
	solve $\Pc 2$, but clearly $LHL$ is non-invertible due to the
	smallest eigenvalue of $L$ fixed at zero, a manifestation of
	the equality constraint in the original problem $\Pc 1$.
	
	To reconcile with this, we change the coordinates of $z$ by the orthogonal matrix $T$ given as  
	\begin{equation*}
	T\hspace{-1mm} = \hspace{-1mm}\begin{bmatrix}
	n\hspace{-1mm} -\hspace{-1mm} 1 + \hspace{-1mm}\sqrt{n} & -1 & \cdots & -1 & \dfrac{1}{\sqrt{n}} \\
	-1 & \ddots & \cdots & \vdots & \\
	\vdots & & \ddots & -1 & \vdots \\
	-1 & \cdots & -1 & n\hspace{-1mm} -\hspace{-1mm} 1 + \hspace{-1mm}\sqrt{n} & \\
	-1 - \sqrt{n} & \cdots & \cdots & -1 - \sqrt{n} & \dfrac{1}{\sqrt{n}}
	\end{bmatrix}\hspace{-1mm}\text{diag}(\hspace{-.5mm}\begin{bmatrix}
	\rho \\ 1
	\end{bmatrix}\hspace{-.5mm}),
	\end{equation*}
	where $\rho = \sqrt{n(n+1+2\sqrt{n})}^{-1}\mathbf{1}_{n-1}$. 
	This transformation leaves the last column as a constant multiple of the
	all ones vector. The new coordinates allow us to rewrite $Lz$ with a
	reduced order variable $\hat{z}\in\real^{n-1}$. More specifically, $Lz =
	LTJ^\top\hat{z}$, where $J = \begin{bmatrix} I_{n-1} &
	\mathbf{0}_{n-1}\end{bmatrix}$. 
	We then rewrite $f$ by substituting $z$ by $\hat{z}$
	\begin{align*}
	f(x^k + Lz)&
	= \dfrac{1}{2} z^\top LHL z + \dfrac{1}{2}x^{k\top} H x^k + z^\top LHx^k					\\
	&+ z^\top Lb + x^{k\top}b        	\\
	&= \dfrac{1}{2} \hat{z}^\top JT^\top LHL T J^\top\hat{z} + \dfrac{1}{2}x^{k\top} H x^k \\
	&+ \hat{z}^\top J T^\top
	(Lb + LHx^k) + x^{k\top}b.
	\end{align*}
	Taking the gradient and Hessian with respect to $\hat{z}$ in this
	expression gives
	\begin{equation*}
	\begin{aligned}
	\nabla_{\hat{z}} f(x^k + Lz) &= 
	JT^\top
	LHLTJ^\top\hat{z} + JT^\top (Lb + LHx^k) \\
	\nabla_{\hat{z}}^2 f(x^k + Lz) &= JT^\top LHLTJ^\top.
	\end{aligned}
	\end{equation*}
	Notice that the zero eigenvalue of $\nabla_z^2 f$ is eliminated by the
	projection. The Newton step associated with $\hat{z}$ is then well
	defined as ${\subscr{\hat{z}}{nt} =-(\nabla_{\hat{z}}^2 f)^{-1}
		\nabla_{\hat{z}} f}$. For brevity, we define $M = JT^\top
	LHLTJ^\top$ and $v(x^k) = Lb+LHx^k$. The Newton step in $\hat{z}$ is
	rewritten as ${\subscr{\hat{z}}{nt} = -M^{-1}(M\hat{z} + JT^\top
		v(x^k))}$ which, when evaluated at $z = 0$, reduces to
	$\subscr{\hat{z}}{nt} = -M^{-1}JT^\top v(x^k)$.
	
	Now, consider a $q$-approximation of $M^{-1}$ given by
	$\widetilde{M_q^{-1}} = \sum_{p=0}^q (I_{n-1} - M)^p$. We can then
	approximate $\subscr{\hat{z}}{nt}$ as $-\widetilde{M_q^{-1}}JT^\top
	v(x^k)$. Returning to the original coordinates:
	\begin{equation*}
	L\subscr{z}{nt} \approx L\subscr{\tilde{z}}{nt} 
	= - LTJ^\top \sum_{p=0}^{q} (I_{n-1} - M)^p JT^\top v(x^k).
	\end{equation*}
	With the property that $LTJ^\top JT^\top L = LL$, we can rewrite $L\subscr{\tilde{z}}{nt}$ as
	\begin{equation}
	L\subscr{\tilde{z}}{nt} = -L \sum_{p=0}^q (I_n - LHL)^p (Lb +LHx^k).
	\label{eq:znewton-step}
	\end{equation}
	
	The series $L(I_n - LHL)^p$ converges with $p\rightarrow\infty$ if the following assumption holds.
	
	\begin{assump}\longthmtitle{Convergent eigenvalues}\label{ass:e-vals}
		The smallest $n-1$ eigenvalues of $(I_n - LHL)$ are contained in 
		the unit ball, i.e. $\exists \ \varepsilon < 1$ such that
		\begin{equation*}
		-\varepsilon I_{n-1} \preceq (I_{n-1} - M) \preceq \varepsilon I_{n-1}.
		\end{equation*}
	\end{assump}
	
	In the following section, we address Assumption~\ref{ass:e-vals} by
	minimizing $\epsilon$ via weight design of the Laplacian. By doing
	this, we aim to obtain a good approximation of $M^{-1}$ from the
	Taylor expansion with small $q$.
	
	\subsection{Weight Design of the
		Laplacian} \label{sec:topology-design}
	
	Our approach returns to the intuition on the rate of convergence of
	the $q$-approximation of $M^{-1}$. We design a weighting scheme for a
	communication topology characterized by $L$ which lends itself to a
	scalable, fast approximation of $\subscr{z}{nt}$. To do this, we aim
	to minimize $\max \vert 1 - \lambda_i(M) \vert$, and pose this problem
	as
	\begin{equation*}		%\label{eq:eval problem}
	\begin{aligned}
	\Pc 3: \quad
	& \underset{\varepsilon, L}{\text{min}}
	&&\varepsilon \\
	& \text{s.t.}
	&&- \varepsilon I_{n-1} \preceq I_{n-1} - M \preceq \varepsilon I_{n-1}, \\
	&&& L \mathbf{1}_n = \mathbf{0}_n, \ L \succeq 0, \\
	&&& L_{ij} = 0,\, j \notin \nodes_i.
	\end{aligned}
	\end{equation*}
	
	There is one nonconvexity in this problem given by the right side of
	the matrix inequality, which is bilinear in $L$.  If we ignore the
	sparsity constraint on $L$, it is possible to treat the constraint as
	an LMI in $M$, obtain an exact solution, and then perform some
	eigendecomposition on $M$ and compute a nonsparse solution in $L$.
	However, implementing this solution would be non-distributed. There
	are path-following techniques available to solve bilinear problems of
	this form~\cite{AH-JH-SB:99}, but simulation results do not produce a
	satisfactory solution via this method. For this reason, we develop a
	convex approximation of $\Pc 3$.
	
	Consider $\varepsilon_-$ and $\varepsilon_+$
	corresponding to the left-hand (convex) and right-hand
	(nonconvex) matrix inequalities,
	respectively. To handle $-\varepsilon_-I_{n-1} \preceq I_{n-1} - M$,
	we write $L$ as a weighted product
	of its incidence matrix, $L = E^\top X E$. Applying Lemma~\ref{lem:schur-comp} makes the constraint become
	\begin{equation}\label{ineq:convex_constraint}
	\begin{bmatrix}
	(\varepsilon_- + 1)I_{n-1} & JT^\top E^\top X E 			\\
	E^\top X E T J^\top & H^{-1}
	\end{bmatrix} \succeq 0.
	\end{equation}
	As for the right-hand side $I_{n-1} - M \preceq \varepsilon_+ I_{n-1}$,
	consider the approximation $LHL \approx \left(\dfrac{\sqrt{H}L +
		L\sqrt{H}}{2}\right)^2$. Substitute this in $M$ to get
	\begin{align*}
	&\dfrac{1}{4}J T^\top
	(\sqrt{H}L + L\sqrt{H})^2
	T J^\top \succeq (1 - \varepsilon_+)I_{n-1} \\
	&\dfrac{1}{2}J T^\top
	(\sqrt{H}L + L\sqrt{H}) T J^\top \succeq \sqrt{(1 - \varepsilon_+)}I_{n-1} \\
	&\dfrac{1}{2}J T^\top
	(\sqrt{H}L + L\sqrt{H}) T J^\top \hspace{-1mm}\succeq (1\hspace{-.5mm} - \hspace{-.5mm}
	\dfrac{\varepsilon_+}{2}\hspace{-.5mm} +\hspace{-.5mm}\dfrac{\varepsilon_+^2}{8} +
	O(\varepsilon_+^3))I_{n-1},
	\end{align*}
	where the second line uses the property that $A^2 \succeq (1 - \varepsilon_+)I_{n-1} \succeq 0
	\Leftrightarrow A \succeq \sqrt{1-\varepsilon_+}I_{n-1} \succeq 0$~\cite{JV-RB:00}, 
	and that $T J^\top J T^\top = I_n - \mathbf{1}_n\mathbf{1}_n^\top /
	n$ is idempotent. The third line expresses the right-hand side as a
	Taylor expansion about $\varepsilon_+ = 0$. Neglecting the higher
	order terms in $\varepsilon_+$, and applying Lemma~\ref{lem:schur-comp} gives
	\begin{equation}\label{ineq:nonconvex_constraint}
	\begin{bmatrix}
	\bullet & \dfrac{1}{\sqrt{8}}\varepsilon_+ I_{n-1} 			\\
	\dfrac{1}{\sqrt{8}}\varepsilon_+ I_{n-1} & I_{n-1}
	\end{bmatrix} \succeq 0,
	\end{equation}
	where ${\bullet = \dfrac{1}{2}J T^\top (\sqrt{H}L + L\sqrt{H})
		T J^\top} - {(1 -
		\dfrac{1}{2}\varepsilon_+)I_{n-1}}$.
	
	Returning to $\Pc 3$, note that the latter three 
	constraints are addressed by writing $L = E^\top X
	E$. Then, the approximate reformulation of $\Pc 3$ can be
	written as
	
	\begin{equation*}		%\label{eq:approx eval problem}
	\begin{aligned}
	\Pc 4: \quad &\underset{\varepsilon_-, \varepsilon_+,
		X}{\text{min}}
	&&\max(\varepsilon_-, \varepsilon_+) \\
	& \text{s.t.}
	&& \varepsilon_- \geq 0, \varepsilon_+ \geq 0,			\\
	&&& X \succeq 0, \eqref{ineq:convex_constraint}, \eqref{ineq:nonconvex_constraint}.	\\
	\end{aligned}
	\end{equation*}
	
	This is a convex problem in $X$ and solvable in polynomial time. To improve
	the solution, we perform some post-scaling. Take $L^\star_0 = E^\top
	X_0^\star E$, where $X_0^\star$ is the solution to $\Pc 4$, and
	${M^\star_0 = JT^\top
		L^\star_0HL^\star_0TJ^\top} \succ
	0$. Then, consider
	\begin{equation*}
	\beta = \sqrt{\dfrac{2}{\lambda_1(M^\star_0) + \lambda_{n-1}(M^\star_0)}}
	\end{equation*}
	and take $L^\star = \beta L^\star_0$. This shifts the eigenvalues of
	$M^\star_0$ to $M^\star$ (defined similarly via $L^\star$) such that
	$1 - \lambda_1(M^\star) = -(1 - \lambda_{n-1}(M^\star))$, which
	shrinks $\max(\vert 1 - \lambda_i(M^\star)\vert)$. We refer to this
	metric as $\varepsilon_{L^\star} := \max(\vert 1 - \lambda_i(M^\star)\vert)$. 
	In fact, it can be verified that this post-scaling
	guarantees $\varepsilon_{L^\star}$ satisfies Assumption~\ref{ass:e-vals}. Then, the solution to $\Pc 4$ followed by a post scaling by
	$\beta$ given by $L^\star$ is an approximation of the solution to the
	nonconvex problem $\Pc 3$ with the sparsity structure preserved.
	
	\subsection{A Bound on Performance}
	
	It should be unsurprising that, given a sparsity structure for the
	network, even the globally optimal solution to $\Pc 3$ will typically
	produce a nonzero $\varepsilon$. We are then motivated to find a
	"best-case scenario" for our solution given the structure of the
	network. The approach for this is straightforward: instead of solving
	$\Pc 3$ for $L$, we solve it for some $A$ where $A$ has the sparsity
	structure of $LHL$, i.e. the two-hop neighbor structure of the
	network. Define $M_A := JT^\top A T J^\top$. This problem is stated
	as:
	\begin{equation*}\label{eq:lower bound}
	\begin{aligned}
	\Pc 5: \quad &\underset{\varepsilon, A}{\text{min}}
	&&\varepsilon \\
	& \text{s.t.}
	&&- \varepsilon I_{n-1} \preceq I_{n-1} - M_A \preceq \varepsilon I_{n-1}, \\
	&&& A \mathbf{1}_n = \mathbf{0}_n, \ A \succeq 0, \\
	&&& A_{ij} = 0, j \notin \nodes^2_i.
	\end{aligned}
	\end{equation*}
	
	This problem is convex in $A$ and produces a solution $\varepsilon_A$,
	which serves as a lower bound for the solution to $\Pc 3$. Of course,
	there may not exist an $L$ with the desired sparsity such that $LHL =
	A$, $\forall A$ in the feasibility set of $\Pc 5$. This is what makes
	$\Pc 3$ difficult to solve, but it gives us some notion for how
	effective the post-scaled solution to $\Pc 4$ is in the following
	sense: $\varepsilon_{L^\star} - \varepsilon_A$ indicates how close
	$\varepsilon_{L^\star}$ is to the lower bound of the global optimum of
	$\Pc 3$.
	
	\subsection{The $\distnewton$ Algorithm}
	
	We now have the tools that we need to introduce the $\dana$, or the
	$\distnewton$ algorithm.
	
	%% Tor: I had to comment out these lines for now to get the document to compile on the new computer -- getting a "No counter 'Algorithm' defined" error
	\begin{algorithm}
		\caption{$\distnewton$}
		\label{Approx Newton}
		\begin{algorithmic}[1]\label{alg:approx-newton}
			\Require $L_j, a_j, b_j$ for $j \in \{i\} \cup \nodes_i$ and communication with nodes $j \in \nodes_i \cup \nodes_i^2$
			\Procedure{Newton}{$x^0,L,H,b,q$}
			\State Initialize $x^0$
			\For{$k = 0, 1, \dots$}
			\ForAll{$i$}				\label{alg1:loop first term}
			\State Acquire $x_j^k$ for $j \in \nodes_i$
			\State $y_i \gets L_ib + (LH)_ix^k$
			\State $z_i \gets -y_i$
			\EndFor
			\For{$p = 1, \dots, q$}		\label{alg1:loop pth term}
			\ForAll{$i$}
			\State Acquire $y_j$ for $j \in \nodes_i^{2}$
			\State $w_i = (I_n - LHL)_i y$
			\EndFor
			\State $y \gets w$
			\State $z \gets z - y$	\label{alg1:sum z}
			\EndFor
			\ForAll{$i$}					\label{alg1:loop step}
			\State Acquire $z_j$ for $j \in \nodes_i$
			\State Compute $x_i^{k+1} = x_i^{k} + \alpha L_i z$
			\EndFor
			\EndFor
			\State \textbf{return} $x$
			\EndProcedure
		\end{algorithmic}
	\end{algorithm}

	The algorithm is constructed directly from~\eqref{eq:short algorithm}
	and~\eqref{eq:znewton-step}. The right-hand factor
	of~\eqref{eq:znewton-step} is computed first in the loop starting on
	line~\ref{alg1:loop first term}. Then, each additional term of the sum
	is computed recursively in the loop starting on line~\ref{alg1:loop
		pth term}, where $w$ is used as an intermediate variable. The outer
	loop of the algorithm is performed starting on line~\ref{alg1:loop
		step}. The process then repeats for a desired number of time steps
	$k$. If $q$ is increased, it requires additional inner-loops and
	two-hop communications, but the step approximation will be more
	accurate.

	It may come as a surprise that the effect of
	increasing the inner loop parameter $q$ by $1$ is algebraically
	equivalent to taking an additional outer loop step at $q = 0$. To
	see this, consider $\hat{x}^{k+1} = x^k + L\subscr{\tilde{z}}{nt}^+
	= x^k + L\left(\subscr{\tilde{z}}{nt} - (I - LHL)^{q+1}(Lb +
	LHx^k)\right)$. The polynomial part of the additional term can be
	quickly expanded with the coefficients given by the
	$\supscr{(q+1)}{th}$ row of Pascal's triangle, which then gets
	multiplied from the right by $Lb$ and $LHx^k$. Compare this to the
	two-step iteration: $x^{k+1} = x^k + L\subscr{\tilde{z}}{nt}$
	followed by $x^{k+2} = x^{k+1} - L(Lb + LHx^{k+1})$, i.e. two outer
	loops with the second loop using $q=0$. It is the case that
	$\hat{x}^{k+1} = x^{k+2}$ and these implementations give an
	equivalent result.
	
	This has some interesting consequences. Applying this
	"loop conversion" notion recursively, it follows that an
	implementation using some $q_a$ inner loop parameter with $k_a$
	outer loops can be converted to an equivalent implementation with
	$k_0 = (q_a + 1)k_a$ outer loops and $q_0 = 0$ inner loop
	parameter. Then, this can be converted again to some $q_b$
	implementation with $k_b = k_0 / (q_b + 1)$ outer loops. If $q_b +
	1$ is not a divisor of $k_0$, the "remainder" can be made up with an
	additional outer loop using an appropriate $q_c < q_b$ parameter.

	A concise summary can be made with the following
	relation: if $k_a, q_a, k_b$ and $q_b$ are such that
	\begin{equation}
	k_a(q_a + 1) = k_b(q_b + 1)
	\end{equation}
	holds, then implementing the $\dana$  from the same $x^0$ will result
	in $x^{k_a} = x^{k_b}$. With this in mind, if direct two-hop
	communications are not available, message passing is required and it
	is less desirable to implement inner-loop iterations when
	communications are failing or lossy. Future work will be devoted to
	evaluate the performance of the alternative approaches with various
	$q$ under the presence of message failures.

	\subsection{Convergence Analysis}
	
	This section establishes convergence properties of the $\distnewton$
	algorithm for problems of the form $\Pc 1$.
	
	\begin{thm}\longthmtitle{Convergence of  $\dana$}
		{\rm If Assumption~\ref{ass:conn-graph}, on the bidirectional
			connected graph, Assumption~\ref{ass:initlal}, on the feasibility
			of the initial condition, Assumption~\ref{ass:cost}, on quadratic
			cost functions, and Assumption~\ref{ass:e-vals}, on convergent
			eigenvalues, hold, then the optimal solution $x^\star$ of $\Pc 1$
			is a unique globally exponentially stable point under the dynamics
			of the $\distnewton$ algorithm for any $q \in \natural$.
		}
	\end{thm}
	
	\begin{proof} 
		For uniqueness, consider the KKT condition which gives the linear
		relation:
		
		\begin{equation*}
		\begin{pmatrix}
		H & \mathbf{1}_n \\ \mathbf{1}^\top_n & 0
		\end{pmatrix} 
		\begin{pmatrix} x^\star \\ \lambda^\star \end{pmatrix} = \begin{pmatrix} -b \\ d \end{pmatrix},
		\end{equation*} 
		where $\lambda^\star$
		is the optimal Lagrange multiplier. It is easy to verify that the
		left-hand side matrix is nonsingular, which gives a unique solution
		$x^\star$.
		
		We establish convergence via a discrete-time 
		Lyapunov function and assuming a unit step size without loss of
		generality. Consider the symmetric matrix $L^\dagger$ characterized by $LL^\dagger
		= I_n - \mathbf{1}_n\mathbf{1}_n^\top /n,$ $\nulo(L^\dagger) =
		\spn(\mathbf{1}_n)$.  Let ${V(x^k) = (x^k - x^\star)^\top L^\dagger
			L^\dagger(x^k - x^\star)}$, which is well defined for any $x \in
		\real^n$, $\sum_i x_i = d$. The dynamics of $\dana$ 
		are given 
		by $x^{k+1} = x^k + L\subscr{\tilde{z}}{nt}$, where
		$\subscr{\tilde{z}}{nt} = -\sum_{p=0}^q (I_n - LHL)^p v(x^k)$, where
		recall that $v(x^k) =Lb +
		LHx^k$. 
		We have $V(x^k) > 0, $ $\ \forall x^k$ such that $\sum_i x_i^k = d,
		x^k \neq x^\star$ (which follows from $(x^k - x^\star) \notin
		\nulo(L^\dagger)$), $V$ is radially unbounded over the space where
		$\sum_i x_i = d$, 
		and $V(x^\star) = 0$. Hence, $V$ is a Lyapunov function candidate. We
		aim to show ${V(x^{k+1}) - V(x^k) < 0}$, for any $q \in \natural$ and $k \ge 0$, where
		\begin{equation}\label{eq:V-znt}
		\begin{aligned}
		&V(x^{k+1}) - V(x^k) = V(x^k + L\subscr{\tilde{z}}{nt}) -
		V(x^k) \\ &= 2\subscr{\tilde{z}}{nt}^\top L L^\dagger
		L^\dagger (x^k - x^\star) + \subscr{\tilde{z}}{nt}^\top (I_n
		- \mathbf{1}_n\mathbf{1}_n^\top /n)\subscr{\tilde{z}}{nt}.
		\end{aligned}
		\end{equation}	
		Since we consider optimization with quadratic cost functions, an exact
		Newton step solves the problem in one step. In other words, $x^k -
		x^\star = -L\subscr{z}{nt} = L\sum_{p=0}^\infty(I_n - LHL)^pv(x^k)$
		holds. Next, define $A_1 := \sum_{p=0}^q(I_n - LHL)^p$ and $A_2 :=
		\sum_{p=0}^\infty(I_n - LHL)^p$, for brevity. Note that $A_2$ is not a
		convergent series expansion due to the non-converging mode, but the
		argument given in Section~\ref{ssec:char_approx_newton} 
		as to "canceling" this mode via multiplication
		by $L$ is applicable. To see that $A_1 \succ 0$ for $q$ odd, we write
		$A_1 = \sum_{l=0}^{(q-1)/2} (2I_n - LHL)(I_n - LHL)^{2l}$. 
		Note that each term in $A_1$ is symmetric and positive
		definite due to $(2I_n - LHL) \succ 0, (I_n - LHL)^{2l} \succ 0$, and
		the product being symmetric. For the case of $q$ even, the last term
		is $(I_n - LHL)^q \succ 0$.
		
		Finally, note that $LL^\dagger = LL^\dagger L^\dagger L = (I_n - \mathbf{1}_n\mathbf{1}_n^\top /n)$ is idempotent. Returning to the matter at hand, substituting in~\eqref{eq:V-znt} gives
		\begin{equation}\label{eq:V-A1-A2}
		\begin{aligned}
		V(x^{k+1}) - V(x^k) &= -2v(x^k)^\top A_1 (I_n -
		\mathbf{1}_n\mathbf{1}_n^\top /n) A_2 v(x^k) \\ &+ v(x^k)^\top A_1
		(I_n - \mathbf{1}_n\mathbf{1}_n^\top /n) A_1 v(x^k).
		\end{aligned}
		\end{equation}
		We have that $v(x^k) \perp \mathbf{1}_n$, so showing the negativity of~\eqref{eq:V-A1-A2} is equivalent to showing
		\begin{multline}\label{eq:Lyap-asym-LMI}
		2A_1 (I_n - \mathbf{1}_n\mathbf{1}_n^\top /n) A_2 - A_1 (I_n -
		\mathbf{1}_n\mathbf{1}_n^\top /n) A_1 \\ \succeq c(I_n -
		\mathbf{1}_n\mathbf{1}_n^\top /n),
		\end{multline}
		for some $c > 0$. To show this, rewrite $A_2 = \sum_{l=0}^\infty (I_n -
		LHL)^{l(q+1)} A_1$. Then,~\eqref{eq:Lyap-asym-LMI} can be written as
		\begin{multline} \label{eq:Lyapunov-LMI-odd}
		A_1(I_n - \mathbf{1}_n\mathbf{1}_n^\top /n)\left( \sum_{l=0}^\infty \Big[ 2(I_n - LHL)^{l(q+1)}\Big] - I_n \right) A_1 \\
		\succeq c(I_n - \mathbf{1}_n\mathbf{1}_n^\top /n).
		\end{multline}
		For $q$ odd, it is straightforward to see this holds. The $-I_n$
		cancels with an $l=0$ term and the remaining terms are raised to even
		powers. Therefore the left-hand side is positive definite and such $c$
		can be found. For $q$ even, we use the bound
		\begin{equation} \label{eq:Lyapunov-LMI-even}
		\begin{aligned}
		\bullet \ &\succeq A_1(I_n - \mathbf{1}_n\mathbf{1}_n^\top /n)\left( \sum_{l=0}^\infty \left[ 2(-\varepsilon)^{l(q+1)} \right] - 1\right)A_1 \\
		&= A_1(I_n - \mathbf{1}_n\mathbf{1}_n^\top /n)\left( \dfrac{2}{1+\varepsilon^{q+1}}-1\right) A_1 \\
		&\succeq c(I_n - \mathbf{1}_n\mathbf{1}_n^\top /n),
		\end{aligned}
		\end{equation}
		where $\bullet$ is the top line of~\eqref{eq:Lyapunov-LMI-odd}.  To
		see this, write $I_n - LHL = QDQ^{-1}$ as its Jordan decomposition,
		where $D$ is a matrix consisting of the Jordan blocks of $I_n - LHL$
		whose diagonal entries are given by the vector $\diag{D} = (1, \eta_1,
		\dots \eta_{n-1})$ where $\eta_i \in [-\varepsilon, \varepsilon], $
		for $i = 1, \dots, n-1$. 
		Then, $(I_n - LHL)^{l(q+1)} =
		QD^{l(q+1)}Q^{-1}$ with $D^{l(q+1)}$ upper triangular. The smallest
		diagonal entry of $\sum_{l=0}^\infty D^{l(q+1)}$ is bounded from below
		by $\sum_{l=0}^\infty (-\varepsilon)^{l(q+1)}$. Finally, we arrive at
		the relation
		\begin{multline*}
		(I_n - \mathbf{1}\mathbf{1}_n^\top /n)\sum_{l=0}^\infty 2QD^{l(q+1)}Q^{-1} \\
		\succeq (I_n - \mathbf{1}\mathbf{1}_n^\top /n)\sum_{l=0}^\infty 2(-\varepsilon)^{l(q+1)},
		\end{multline*}
		which gives the bound in~\eqref{eq:Lyapunov-LMI-even}.
		%The bound can be seen via Jordan decomposition of $(I_n - LHL)$, and 
		From Assumption~\ref{ass:e-vals} we have $\varepsilon \in [0,1)$. This verifies~\eqref{eq:Lyapunov-LMI-odd} ($q$ odd) or~\eqref{eq:Lyapunov-LMI-even} ($q$ even) hold for some $c >
		0$. This is sufficient to show $x^\star$ is a globally asymptotically
		stable point under the dynamics of $\distnewton$ algorithm.
		
		For exponential convergence, we are interested in showing $V(x^{k+1})
		\leq (1-\hat{c})V(x^k)$ for some $\hat{c} \in (0,1)$ that is
		independent of $k$. Consider the $q = 0$ case without loss of
		generality. Note that $v(x^\star) = 0$, implying $v(x^k) = v(x^k) -
		v(x^\star) = LH(x^k - x^\star)$. 
		Then, substituting in~\eqref{eq:V-A1-A2} and using the $c$
		characterized by~\eqref{eq:Lyapunov-LMI-odd}--\eqref{eq:Lyapunov-LMI-even} gives
		\begin{equation}\label{eq:Lyap-exponential}
		\begin{aligned}
		V(x^{k+1}) &\leq V(x^k) -c(v(x^k)^\top LL^\dagger L^\dagger Lv(x^k)) \\
		&= V(x^k) - c(x^k - x^\star)^\top HL^2L^{\dagger 2}L^2H(x^k - x^\star) \\
		&\leq (1 - \hat{c})V(x^k),
		\end{aligned}
		\end{equation}
		where $\hat{c} \in (0,1)$ is a function of $c, L$,
		and $H$, and $\hat{c} > 0$ due to $v(x^k) \perp \nulo{(L)}$.  This
		shows the Lyapunov function exponentially decays at a rate
		characterized by $\hat{c}$.
	\end{proof}

	\section{Simulations and Discussion}		\label{sec:sims-discuss}
	
	In this section, we implement our weight design and verify the
	convergence of the $\distnewton$ algorithm 
	for different $q \ge 0$, and compare their speed of convergence in
	terms of the number of outer loops $k$ employed.  Returning to the
	economic dispatch motivation, we consider a network with $n = 50$
	generators and $\vert \mathcal{E} \vert = 150$ randomly generated
	bidirectional communication links between generators. The local
	computations required of each generator are simple vector operations
	whose dimension scales linearly with the network size, which can be
	implemented on a microprocessor. The graph topology is plotted in
	Figure~\ref{fig:graph}.  The cost coefficients are generated as $a_i
	\in \mathcal{U} \left[ 0.8, 1.2 \right],$ $b_i \in \mathcal{U} \left[
	0, 1 \right],$ and power requirement is taken to be $d = 50$MW. We
	compare to Distributed Gradient Descent (DGD) with unit step size and
	two weightings on $L$.  The first is ``unweighted'', in the sense that
	$L$ is taken to be the degree matrix minus the adjacency matrix of the
	graph, followed by the post-scaling described in
	Section~\ref{sec:topology-design} to guarantee convergence.  The
	second weighting on $L$ for DGD is proposed in~\cite{LX-SB:06}. The
	results are given in Figure~\ref{fig:qcomp}, which shows linear
	convergence to the optimal value as the number of iterations
	increases, with fewer iterations needed for larger $q$. We note a
	substantially improved outer-loop convergence over the DGD methods,
	even for the $q=0$ case which utilizes an equal number of
	agent-to-agent communications as DGD.

	\begin{figure}[h]
		\centering
		\includegraphics[scale = 0.3]{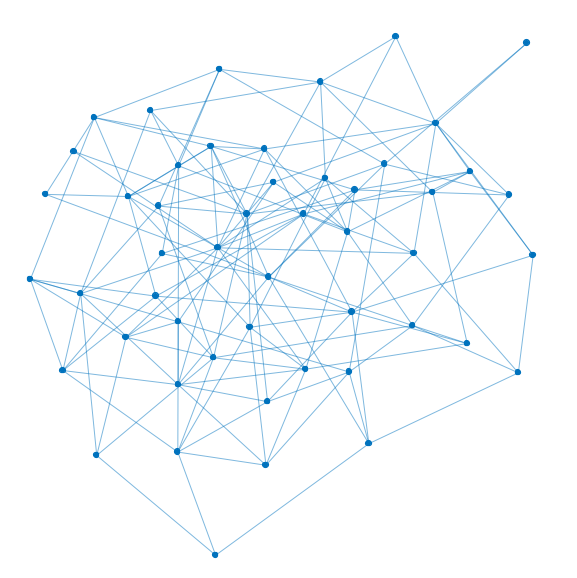}
		\caption{Communication topology used for comparisons of $\distnewton$ and DGD.}
		\label{fig:graph}
	\end{figure}
	
	\begin{figure}[h]
		\centering
		\includegraphics[scale = 0.55]{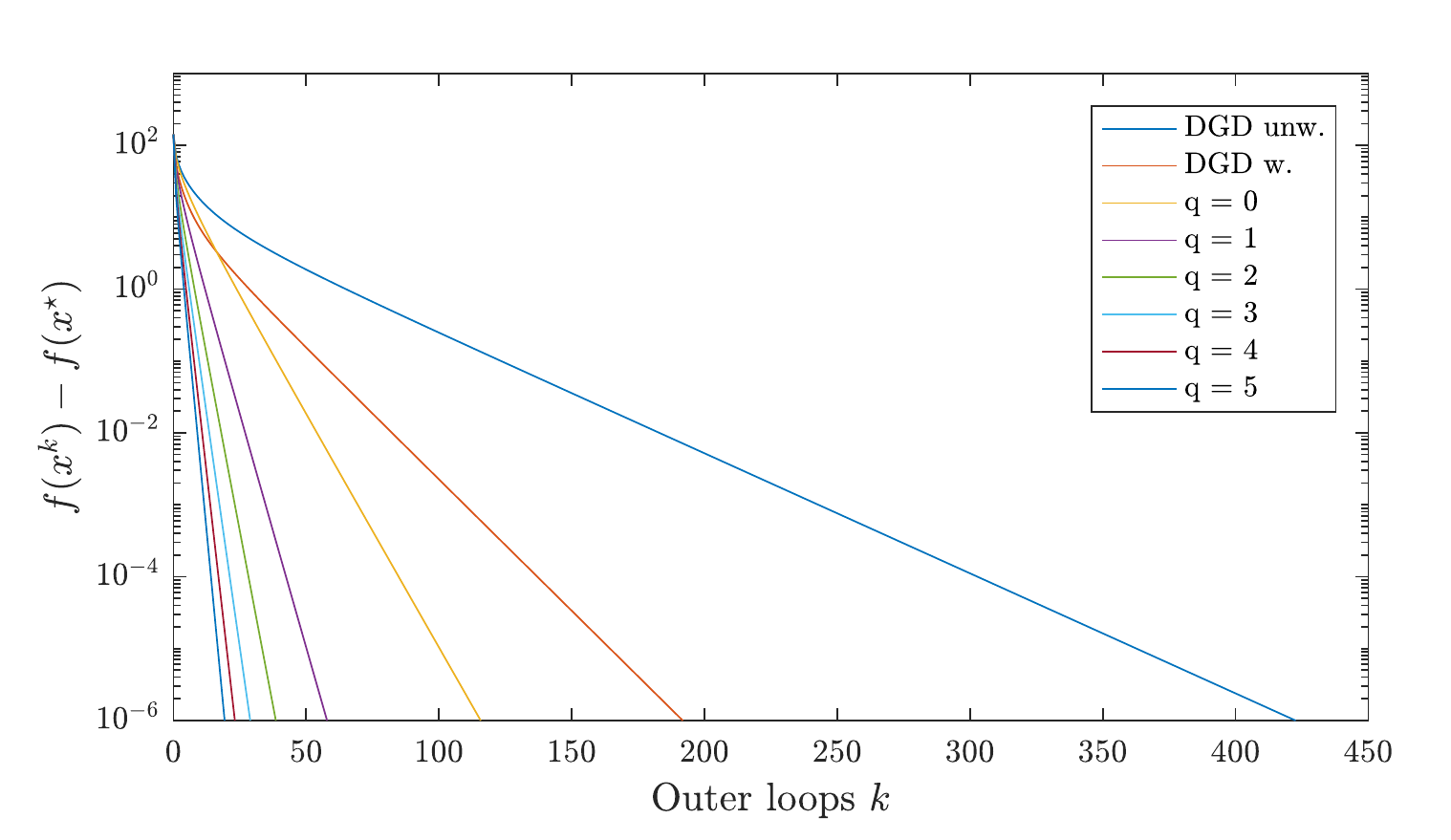}
		\caption{Convergence of $\distnewton$ for varying
			$q$ compared to unweighted and weight-designed DGD.}
		\label{fig:qcomp}
	\end{figure}
	
	On the weighting design side, consider the following metrics.  The
	solution to $\Pc 4$ followed by the post-scaling by $\beta$ gives
	$\varepsilon_{L^\star} := \max(\vert 1 - \lambda_i(M^\star)\vert )$;
	this metric represents the convergence speed of $\distnewton$ 
	when applying our proposed weight design of $L$. 
	Using the same topology $(\nodes, \mathcal{E})$, the solution to $\Pc
	5$ gives the metric $\varepsilon_A$. Note that $\varepsilon_A$ is a
	\emph{best-case} estimate of the weight design problem, which cannot
	be implemented in practice, while $\varepsilon_{L^\star}$ is the
	metric for which we can compute an $L^\star$. The objective of each
	problem is to minimize the associated
	$\varepsilon$; 
	to this end, we aim to characterize the
	relationship between network parameters and these metrics. We ran 100
	trials on each of 16 test cases which encapsulate a variety of
	parameter cases: two cases for the cost coefficients, a \textit{tight}
	distribution $a_i \in \mathcal{U}\left[ 0.8, 1.2\right]$ and a
	\textit{wide} distribution $a_i \in \mathcal{U}\left[ 0.2,
	5\right]$. For topologies, we randomly generated connected graphs
	with network size $n \in \{10, 20, 30, 40, 50\}$, a \textit{linearly}
	scaled number of edges $\vert\mathcal{E}\vert = 3n$, and a
	\textit{quadratically} scaled number of edges $\vert\mathcal{E}\vert =
	0.16n^2$ for the $n \in \{30, 40, 50\}$ cases. The linearly scaled
	connectivity case corresponds to keeping the average degree of a node
	constant for increasing network sizes, while the quadratically scaled
	case roughly preserves the proportion of connected edges to total
	possible edges, which is a quadratic function of $n$ and equal to
	$n(n-1)/2$ for an undirected network. The results are depicted in
	Table~\ref{L-design-table}. This gives the mean $\mu$ and standard
	deviation $\sigma$ of the distributions for \emph{performance}
	$\varepsilon_{L^\star}$ and \emph{performance gap} $\varepsilon_{L^\star} -
	\varepsilon_A$.
	
	\begin{table}[]
		\centering
		\caption{Laplacian Design}
		\label{L-design-table}
		\begin{tabular}{|c||c|c|c|c|}
			\hline
			\begin{tabular}[c]{@{}c@{}} $a_i \in \mathcal{U} \left[ 0.8, 1.2\right]$ \\ $b_i \in \mathcal{U} \left[ 0, 1\right]$\end{tabular} & $\mu ( \varepsilon_{L^\star})$ & $\sigma (\varepsilon_{L^\star})$ & $\mu ( \varepsilon_{L^\star} - \varepsilon_A)$ & $\sigma ( \varepsilon_{L^\star} - \varepsilon_A)$ \\ \hline
			\begin{tabular}[c]{@{}c@{}}$n = 10$ \\ $\vert \mathcal{E} \vert = 30$ \end{tabular}  & 0.6343 & 0.0599 & 0.2767 & 0.0186 \\ \hline
			\begin{tabular}[c]{@{}c@{}}$n = 20$ \\ $\vert \mathcal{E} \vert = 60$ \end{tabular}  & 0.8655 & 0.0383 & 0.2879 & 0.0217 \\ \hline
			\begin{tabular}[c]{@{}c@{}}$n = 30$ \\ $\vert \mathcal{E} \vert = 90$ \end{tabular}  & 0.9100 & 0.0250 & 0.2666 & 0.0233 \\ \hline
			\begin{tabular}[c]{@{}c@{}}$n = 40$ \\ $\vert \mathcal{E} \vert = 120$ \end{tabular}  & 0.9303 & 0.0201 & 0.2501 & 0.0264 \\ \hline
			\begin{tabular}[c]{@{}c@{}}$n = 50$ \\ $\vert \mathcal{E} \vert = 150$ \end{tabular}  & 0.9422 & 0.0175 & 0.2375 & 0.0264 \\ \hline
			\begin{tabular}[c]{@{}c@{}}$n = 30$ \\ $\vert \mathcal{E} \vert = 144$ \end{tabular}  & 0.7266 & 0.0324 & 0.2973 & 0.0070 \\ \hline
			\begin{tabular}[c]{@{}c@{}}$n = 40$ \\ $\vert \mathcal{E} \vert = 256$ \end{tabular}  & 0.6528 & 0.0366 & 0.2829 & 0.0091 \\ \hline
			\begin{tabular}[c]{@{}c@{}}$n = 50$ \\ $\vert \mathcal{E} \vert = 400$ \end{tabular}  & 0.5840 & 0.0281 & 0.2641 & 0.0101 \\ \hline \hline
			\begin{tabular}[c]{@{}c@{}} $a_i \in \mathcal{U} \left[ 0.2, 5\right]$ \\ $b_i \in \mathcal{U} \left[ 0, 1\right]$\end{tabular} & $\mu ( \varepsilon_{L^\star})$ & $\sigma (\varepsilon_{L^\star})$ & $\mu ( \varepsilon_{L^\star} - \varepsilon_A)$ & $\sigma ( \varepsilon_{L^\star} - \varepsilon_A)$ \\ \hline
			\begin{tabular}[c]{@{}c@{}}$n = 10$ \\ $\vert \mathcal{E} \vert = 30$ \end{tabular}  & 0.6885 & 0.0831 & 0.3288 & 0.0769 \\ \hline
			\begin{tabular}[c]{@{}c@{}}$n = 20$ \\ $\vert \mathcal{E} \vert = 60$ \end{tabular}  & 0.8965 & 0.0410 & 0.3241 & 0.0437 \\ \hline
			\begin{tabular}[c]{@{}c@{}}$n = 30$ \\ $\vert \mathcal{E} \vert = 90$ \end{tabular}  & 0.9389 & 0.0254 & 0.2878 & 0.0395 \\ \hline
			\begin{tabular}[c]{@{}c@{}}$n = 40$ \\ $\vert \mathcal{E} \vert = 120$ \end{tabular} & 0.9539 & 0.0189 & 0.2830 & 0.0355 \\ \hline
			\begin{tabular}[c]{@{}c@{}}$n = 50$ \\ $\vert \mathcal{E} \vert = 150$ \end{tabular} & 0.9628 & 0.0168 & 0.2590 & 0.0335 \\ \hline
			\begin{tabular}[c]{@{}c@{}}$n = 30$ \\ $\vert \mathcal{E} \vert = 144$ \end{tabular} & 0.7997 & 0.0520 & 0.3587 & 0.0524 \\ \hline
			\begin{tabular}[c]{@{}c@{}}$n = 40$ \\ $\vert \mathcal{E} \vert = 256$ \end{tabular} & 0.7339 & 0.0550 & 0.3688 & 0.0569 \\ \hline
			\begin{tabular}[c]{@{}c@{}}$n = 50$ \\ $\vert \mathcal{E} \vert = 400$ \end{tabular} & 0.6741 & 0.0487 & 0.3543 & 0.0425 \\ \hline
		\end{tabular}
	\end{table}
	
	There are a few notable takeaways from these results. Firstly, we note
	that the \emph{tightly} distributed coefficients $a_i$ result in
	improved $\varepsilon_{L^\star}$ across the board compared to the
	\emph{widely} distributed coefficients. We attribute this to the
	approximation $LHL \approx \left(\dfrac{\sqrt{H}L +
		L\sqrt{H}}{2}\right)^2$ being more accurate for roughly
	homogeneous $H = \diag{a_i}$. Next, it is clear that in the cases with
	\emph{linearly} scaled edges, $\varepsilon_{L^\star}$ worsens as
	network size increases.  This is intuitive: the \textit{proportion} of
	connected edges in the graph decreases as network size increases in
	these cases. This also manifests itself in the performance gap
	$\varepsilon_{L^\star} - \varepsilon_A$ shrinking, indicating the
	\emph{best-case} solution $\varepsilon_A$ (for which a valid $L$ does
	not necessarily exist) degrades even quicker as a function of network
	size than our solution $\varepsilon_{L^\star}$. On the other hand,
	$\varepsilon_{L^\star}$ substantially improves as network size
	increases in the \emph{quadratically} scaled cases, with a roughly
	constant performance gap $\varepsilon_{L^\star} -
	\varepsilon_A$. Considering this relationship between the linear and
	quadratic scalings on $\vert \mathcal{E}\vert$ and the metrics
	$\varepsilon_{L^\star}$ and $\varepsilon_A$, we get the impression
	that both proportion of connectedness and average node degree play a
	role in both the effectiveness of our weight-designed solution
	$L^\star$ and the best-case solution.  For this reason, we postulate
	that $\varepsilon_{L^\star}$ remains roughly constant in large-scale
	applications if the number of edges is scaled subquadratically as a
	function of network size; equivalently, the convergence properties of
	$\distnewton$ algorithm remain relatively unchanged when using our
	proposed weight design and growing the number of communications per
	agent sublinearly as a function of
	$n$. 
	
	\section{Conclusion and Future Work}
	\label{sec:conclusion}
	
	Motivated by economic dispatch problems, this work proposed the novel
	$\distnewton$ algorithm. More generally, the algorithm can be applied
	to a class of separable resource allocation problems with quadratic
	costs. We then posed the topology design proplem, and provided an
	effective method for designing communication weightings.  The
	weighting design we propose is more cognizant of the problem geometry,
	and it outperforms the current literature on network weight
	design. Ongoing work includes the generalization to arbitrary convex
	functions, nonseparable contexts, general equality and inequality
	constraints, design for robustness under uncertain parameters or lossy
	communications, and a more direct application to economic dispatch and
	power networks.  Additionally, we are interested in further studying
	branch and bound methods for solving bilinear problems and other
	existing heuristics for topology design within the proposed framework.

	\bibliographystyle{abbrv}
	\bibliography{alias,SMD-add,SM}

\end{document}

%% file: macros-abbrev.tex
\renewcommand{\natural}{{\mathbb{N}}}

\newcommand{\real}{\ensuremath{\mathbb{R}}}

\newcommand{\until}[1]{\{1,\dots, #1\}}
\newcommand{\subscr}[2]{#1_{\textup{#2}}}
\newcommand{\supscr}[2]{#1^{\textup{#2}}}

\newcommand{\diag}[1]{\operatorname{diag}(#1)}

\newcommand{\spn}{\operatorname{span}}

\newcommand{\graph}{G}

\renewcommand{\epsilon}{\varepsilon}

\newcommand{\Pc}{{\mathcal{P}}}

\usepackage{psfrag}